\documentclass[10pt]{article}
\usepackage{amsfonts}
\usepackage{latexsym}
\usepackage{cite}
\usepackage{amsmath,amsfonts,latexsym,amssymb}
\usepackage[mathscr]{eucal}
\usepackage{cases,color}
\usepackage{amsthm}

\usepackage[bf,small]{caption2}
\usepackage{float}
\usepackage{graphicx}
\usepackage{amsmath}
\usepackage{amssymb}
\usepackage[all]{xy}

\newtheorem{theorem}{theorem}[section]

\newtheorem{lem}[theorem]{Lemma}

\newtheorem{rmk}[theorem]{Remark}
\newtheorem{nota}[theorem]{Notation}

\begin{document}

\title{\textbf{Representation category of the quantum double of ${\rm SL}(2,q)$}}
%\\ : tensor products and braidings}}
\author{\Large Haimiao Chen
\footnote{Email: \emph{chenhm@math.pku.edu.cn}}\\
\normalsize \em{Mathematics, Beijing Technology and Business University, Beijing, China}}
\date{}
\maketitle

\begin{abstract}
  For each odd prime power $q$ with $q\ge 5$ and $4\mid q-1$, we investigate the structure of the representation category of the quantum double of ${\rm SL}(2,q)$, determining its tensor products and braidings.

  %The result is applied to compute untwisted Dijkgraaf-Witten invariants of links.
  \medskip
  \noindent {\bf key words:} quantum double; representation category; ${\rm SL}(2,q)$; Dijkgraaf-Witten invariant; 3-manifold  
\end{abstract}

\section{Introduction}

Let $\Gamma$ be a finite group. For a closed oriented 3-manifold $M$, the untwisted {\it Dikgraaf-Witten invariant} $Z_\Gamma(M)$ counts homomorphisms of $\pi_1(M)$ to $\Gamma$.
%to each closed oriented 3-manifold $M$ is associated a number
%$$Z_\Gamma(M)=\frac{1}{\#\Gamma}\cdot\#\hom(\pi_1(M),\Gamma),$$
%to each closed oriented surface $\Sigma$ is associated a vector space
%$$Z_\Gamma(\Sigma)={\rm Map}(\hom(\pi_1(\Sigma),\Gamma),\mathbb{C})^\Gamma,$$
%where $\Gamma$ acts via conjugate, and to each oriented 3-manifold $M$ with $\partial M=\Sigma\ne\emptyset$ is associated an element
%$Z_\Gamma(M)\in Z_\Gamma(\Sigma)$ whose value at $\phi:\pi_1(\Sigma)\to \Gamma$ is the number of homomorphisms $\pi_1(M)\to\Gamma$ whose
%restriction on $\Sigma$ is equal to $\phi$.
DW theory was first proposed in \cite{DW90} as a 3-dimensional topological quantum field theory (TQFT), and since then has been further
studied by many people; one may refer to \cite{Fe93, FQ93, Fr94, Mo10, Wa92, Ye92}, etc.
%The key ingredient of the untwisted DW theory is counting homomorphisms from the fundamental group of a manifold to $\Gamma$.

%There is a wide class of 3d TQFTs.
Recall \cite{BJ01,RT91} that, given an {\it modular tensor category} (MTC), one can associate
to each ``colored framed link" an invariant, and then use it to define a cobordism invariant, which fits in TQFT axioms.
From $\Gamma$ one can construct the {\it quantum double} $D(\Gamma)$, whose representation category $\mathcal{E}(\Gamma)$ is an MTC,
giving rise to a 3d TQFT $\textrm{RT}_{\Gamma}$. Freed \cite{Fr94-2} proved that $Z_{\Gamma}$ coincides with $\textrm{RT}_{\Gamma}$, by computing the DW invariant of the circle to be $\mathcal{E}(\Gamma)$ and applying Reshetikhin-Turaev Theorem.
This permits us to compute $Z_\Gamma(M)$ for a 3-manifold $M$, from a surgery presentation of $M$.

The fundamental group of a 3-manifold is no doubt important, but as a non-commutative object, it is usually difficult to handle. DW invariant extracts from fundamental group partial information which is manageable. Remarkably, in recent years, many deep connections between 3-manifolds
and finite groups have been revealed (see \cite{BF15,Ko13,LR10,Wi17,WZ17} and the references therein).
Thus it is worth investigating the MTC $\mathcal{E}(\Gamma)$ for various finite groups $\Gamma$. However, till now there is no such work
seen in the literature.
%except for the author's unpublished work \cite{Ch13}.

In this paper, we consider the case $\Gamma={\rm SL}(2,q)$, where $q$ is an odd prime power with $q\ge 5$ and $4\mid q-1$.
The motivation for this choice is two-fold.
Firstly, it is Weil's great insight that, if a variety $X$ can be defined both over $\mathbb{C}$ and $\mathbb{F}_q$, then the behavior of $\#X(\mathbb{F}_q)$ for various $q$ ``remembers" the topology of $X(\mathbb{C})$. In particular,
for a finitely generated group $\Lambda$ and a linear algebraic group $G$ (which can be ${\rm SL}_n$, ${\rm PGL}_n$, etc), topological
properties of the $G(\mathbb{C})$-representation variety of $\Lambda$ are encoded in the numbers $\#\hom(\Lambda,G(\mathbb{F}_q))$. More than expected is established in \cite{BH17} that the E-polynomials of the ${\rm GL}(n,\mathbb{C})$- and ${\rm SL}(n,\mathbb{C})$-character varieties of $\Lambda$ can be computed by counting representations over $\mathbb{F}_q$. Works concerning Fuchsian groups are \cite{LS05, HR08}; zeta-functions of some knot groups were computed in \cite{Si00,Ha11,Ha14}, and \cite{LX03,LX04} were also in this vein. Secondly, for a
hyperbolic 3-manifold, its fundamental group is closely related to ${\rm SL}(2,q)$; see, for instance, \cite{Ha15,LR98,GTM219}.
Significantly, it was discovered in \cite{Ha15} that the Hasse-Weil zeta function of the canonical components of the ${\rm PSL}_2$-character variety of a closed hyperbolic 3-manifold of finite volume is equal to the Dedekind zeta function of its invariant trace field.

%the ${\rm SL}(2)$-representations of weaving knots are extremely complicated. An important motivation is to compute the

%Interestingly as shown in \cite{KS12}, all the knots in Rolfsen's table can be distinguished by counting homomorphisms to ${\rm SL}(2,p)$.

%We shall give a detailed description for $\mathcal{E}({\rm SL}(2,q))$.
A key feature of ${\rm SL}(2,q)$ is that the centralizer of each non-central element is abelian. This makes $\mathcal{E}({\rm SL}(2,q))$ relatively
easy to understand, and enables us to derive explicit formulas for the decomposition of tensor products of irreducible representations.
Based on these, we can compute the braidings.

%The restriction $4\mid q-1$ is not too serious, since from the perspective of computing DW invariants of links, the value of
%$Z_\Gamma(S^3-L)$ at $(x_1,h_1;\ldots,x_n,h_n)$ for $x_1,\ldots,x_n\in{\rm SL}(2,q)$ is the same as $Z_{\widetilde{\Gamma}}(S^3-L)$
%(at the same component), where $\widetilde{\Gamma}={\rm SL}(2,q^2)$, just because $x_1,\ldots,x_n$ generate $\pi_1(S^3-L)$.

We expect these results to be foundational for future researches concerning interplay between arithmetic and hyperbolic 3-manifold.

The content is organized as follows. In Section 2 we recall the construction of quantum double $D(\Gamma)$ and the MTC structure of its representation category $\mathcal{E}(\Gamma)$, for a general finite group $\Gamma$. In Section 3 we give details to ${\rm SL}(2,q)$ and simple objects of $\mathcal{E}({\rm SL}(2,q))$. The main achievements of this paper are presented in Section 4, where we derive explicit formulas
for the tensor products and the braidings. %In the last section, we compute DW invariants for some knots.
%The associators of $\mathcal{E}({\rm PSL}(2,q))$ will be studied in a forthcoming paper.

\medskip

{\bf Acknowledgement}

%This work is supported by NSFC-11401014.
I am grateful to D.S. Freed for beneficial communications on Dijkgraaf-Witten theory, and to Shenghao Sun (at THU) for elaborating Weil Conjecture.

\section{Preliminary} %on quantum double and its representations

\begin{nota}
\rm For elements $x,y$ of some group, let $[x]$ denote the conjugacy class containing $x$, let ${\rm Cen}(x)$ denote the
centralizer of $x$, and let $y.x=yxy^{-1}$.

For a positive integer $m$, let $\mathbb{Z}_m$ denote $\mathbb{Z}/m\mathbb{Z}$.

For a finite field $\mathbb{F}$, let $\mathbb{F}^{\times}$ denote the set of nonzero elements, which forms a cyclic group under multiplication; let $\mathbb{F}^{\times2}$ denote the set of nonzero squares.

Given positive integers $a,b$ with $a<b$, let $[a,b]=\{a,a+1,\ldots,b\}$.
\end{nota}

%\subsection{The category $\mathcal{E}(\Xi)$}

For this section, one may refer to \cite{BJ01} Section 3.2 and \cite{Fr94}.
Given a finite group $\Gamma$, the quantum double $D(\Gamma)$ is a {\it quasi-triangular} Hopf algebra with underlying vector space spanned by $\langle g,x\rangle, g,x\in\Gamma$, and equipped with the following operations:
\begin{itemize}
  \item multiplication: $\langle g,x\rangle\langle h,y\rangle=\delta_{g,x.h}\langle g,xy\rangle$;
  \item unit: $1=\sum_{g\in\Gamma}\langle g,e\rangle$;
  \item comultiplication: $\langle g,x\rangle\mapsto \sum_{g_1g_2=g}\langle g_1,x\rangle\otimes\langle g_2,x\rangle$;
  \item counit: $\langle g,x\rangle\mapsto\delta_{g,e}$;
  \item antipode: $\langle g,x\rangle\mapsto\langle x^{-1}.g^{-1},x^{-1}\rangle$;
  \item R-matrix: $R=\sum_{g,h\in\Gamma}\langle g,e\rangle\otimes\langle h,g\rangle$.
\end{itemize}

Let $\mathcal{E}(\Gamma)$ be the category of finite-dimensional representations of $D(\Gamma)$.
Each object can be identified with a $\Gamma$-graded vector space $U=\bigoplus_{x\in\Gamma}U_{x}$
together with a left action $L^U:\Gamma\times U\to U$ such that $L^U_g(U_{x})=U_{g.x}$ for all $g,x$. A morphism $f:U\rightarrow V$ is a family of linear maps $f=\{f_{x}:U_{x}\rightarrow V_{x}\colon x\in\Gamma\}$
such that $L^U_g\circ f_{x}=f_{g.x}\circ L^U_g$ for all $g,x$.

Here are main ingredients of the MTC structure of $\mathcal{E}(\Gamma)$:
\begin{itemize}
  \item The {\it tensor product}, denoted by $\odot$, is given by
        $$(U\odot V)_{x}=\bigoplus\limits_{x_{1}x_{2}=x}U_{x_{1}}\otimes V_{x_{2}}.$$
        The {\it unit object} is $\underline{\mathbb{C}}$ with
        $\underline{\mathbb{C}}_{1}=\mathbb{C}$ and $\underline{\mathbb{C}}_{x}=0$ for all $x\neq 1$.
  \item The {\it associator} $(U\odot V)\odot W\rightarrow U\odot(V\odot W)$ is given by the natural isomorphism of vector spaces
        $$(U_{x}\otimes V_{y})\otimes W_{z}\cong U_{x}\otimes(V_{y}\otimes W_{z}).$$
  \item The {\it braiding} $R_{U,V}:U\odot V\rightarrow V\odot U$ sends $u\otimes v\in U_x\otimes V_y$ to $L^V_x(v)\otimes u\in V_{x.y}\otimes U_x$
        for all $u\in U_{x}, v\in V_{y}$.
  \item The {\it dual} is given by $(U^{\ast})_{x}=(U_{x^{-1}})^\ast$ (the dual vector space), with
        $L^{U^\ast}_g:(U^{\ast})_{x}\rightarrow (U^{\ast})_{g.x}$ given by the dual map of $g^{-1}:U_{g.x^{-1}}\rightarrow U_{x^{-1}}$. In addition, there are natural morphisms $\iota_{U}:\underline{\mathbb{C}}\rightarrow U\odot U^{\ast}$ and $\epsilon_{U}:U^{\ast}\odot U\rightarrow\underline{\mathbb{C}}$, which respectively can be expressed as
        \begin{align}
        \iota_{U}(1)=\sum\limits_{x}\sum\limits_{a}u_{x,a}\otimes u_{x,a}^\ast,  \qquad
        \epsilon_{U}(u_{x,a}^\ast\otimes u_{y,b})=\delta_{x,y}\delta_{a,b},
        \end{align}
        where $\{u_{x,a}\}$ is an arbitrary basis of $U_{x}$, and $\{u_{x,a}^\ast\}$ is the corresponding dual basis of $(U_{x})^\ast=U^\ast_{x^{-1}}$.
\end{itemize}

For $U\in\mathcal{E}(\Gamma)$, call $\{x\in\Gamma\colon U_x\ne 0\}$ the {\it support} of $U$ and denote it by ${\rm supp}(U)$;
the {\it character} of $U$ is the function
\begin{align*}
\chi_U:\Gamma\times\Gamma\to\mathbb{C}, \qquad
(x,h)\mapsto\begin{cases} 0,&h\notin{\rm Cen}(x), \\ {\rm tr}(h:U_x\to U_{x}),&h\in{\rm Cen}(x). \end{cases}
\end{align*}
and the {\it quantum dimension} of $U$ is defined as
\begin{align*}
d_{U}=\sum\limits_{x\in\Gamma}\chi_{U}(x,1).
\end{align*}
The character satisfies the following:
\begin{enumerate}
  \item[\rm(i)] $U\cong U'$ if and only if $\chi_U=\chi_{U'}$;
  \item[\rm(ii)] $\chi_U(g.x,g.h)=\chi_U(x,h)$ for all $g$;
  \item[\rm(iii)] $\chi_{U^\ast}(x,h)=\chi_{U}(x^{-1},h^{-1})$ for all $x,h$.
\end{enumerate}
%The last equality is due to
%$${\rm tr}((h^{-1}:U_{x^{-1}}\to U_{x^{-1}})^\ast)={\rm tr}(h^{-1}:U_{x^{-1}}\to U_{x^{-1}}).$$

Let $S$ be a system of representatives for conjugacy classes of $\Gamma$.
For each $x\in S$, let $\mathcal{R}_x$ be a system of representatives for isomorphism classes of irreducible representations
of $C(x)$, and let $R_x\subset\Gamma$ be a system of representatives for left cosets of $C(x)$. For $\rho\in\mathcal{R}_x$ with
representation space $U(\rho)$, let $U(x,\rho)$ be the object of $\mathcal{E}(\Gamma)$ with
\begin{align*}
U(x,\rho)_y=\begin{cases} U(\rho),&y\in[x], \\ 0, &y\notin[x]; \end{cases}
\end{align*}
for each $g\in\Gamma$ and each $a\in R_x$, take $a'\in R_x$ and $b\in{\rm Cen}(x)$ such that $ga=a'b$, and define
$L_g^{U(x,\rho)}:U(x,\rho)_{a.x}\to U(x,\rho)_{a'.x}$ by $\rho(b)$. In particular, $U(1,\rho)$ is the same as an irreducible representation of $\Gamma$.

It is known that each simple object of $\mathcal{E}(\Gamma)$ is isomorphic to $U(x,\rho)$ for a unique pair $(x,\rho)$ with $x\in S$ and $\rho\in\mathcal{R}_x$.

For $U=U(x,\rho)$, put
\begin{align}
\theta_{U}={\rm tr}(\rho(x)).
\end{align}

\section{The category $\mathcal{E}({\rm SL}(2,q))$}

\subsection{Preparation}

Let $q=p^n$ with $p$ an odd prime and $n\ge 1$. Suppose $q\equiv 1\pmod{4}$; write $q=4h+1$.

For $j\in\{1,2\}$, let $\overline{j}=3-j$.

For $\mu\in\{+,-\}$, $k\in[1,2h-1]$ and $\ell\in[1,2h]$, set
\begin{align*}
&\overline{\mu}=-, &&\acute{\mu}=0, && k_\mu=k, && \ell^\mu=\ell, && \text{if} \quad \mu=+; \\
&\overline{\mu}=+, &&\acute{\mu}=1, && k_\mu=2h-k, && \ell^\mu=2h+1-\ell, &&\text{if} \quad \mu=-.
\end{align*}

As is well-known,
\begin{align*}
\#{\rm SL}(2,\mathbb{F}_q)=q(q^{2}-1).
\end{align*}

Fix a generator $\check{e}$ of $\mathbb{F}_{q^2}^{\times}$, so that $\check{e}^{(q^2-1)/2}=-1$.
Let
\begin{align}
\tilde{e}=\check{e}^{2h+1}, \qquad e=\tilde{e}^2, \qquad \tilde{f}=\check{e}^{2h}, \qquad f=\tilde{f}^2.
\end{align}
Then $e$ is a generator of $\mathbb{F}_{q}^{\times}$.
We have
$$\mathbb{F}_{q^{2}}=\{a+b\tilde{e}\colon a,b\in\mathbb{F}_{q}\}.$$
For $x=a+b\tilde{e}$, its {\it conjugate} is $\overline{x}:=a-b\tilde{e}=x^q$, and its {\it norm} is
\begin{align*}
\mathcal{N}(x)=x\overline{x}=x^{q+1}=a^2-b^2e.
\end{align*}
Let ${\rm tr}:\mathbb{F}_q\to\mathbb{F}_p$ is the trace. See Section 8 of \cite{Mo96} for definitions. It is known that $\mathcal{N}:\mathbb{F}_{q^{2}}^{\times}\to\mathbb{F}_{q}^{\times}$ is an epimorphism of groups, with kernel generated by $f$, and ${\rm tr}$ is surjective.

For $k,\ell\in\mathbb{Z}$, put
\begin{align}
r_k=\frac{1}{2}(e^k+e^{-k}), \qquad r'_k&=\frac{1}{2}(e^k-e^{-k}), \\
\tilde{s}_\ell=\frac{1}{2}(\tilde{f}^\ell+\tilde{f}^{-\ell}), \qquad \tilde{s}'_\ell&=\frac{1}{2}(\tilde{f}^\ell-\tilde{f}^{-\ell}), \\
s_\ell=\frac{1}{2}(f^\ell+f^{-\ell}), \qquad s'_\ell&=\frac{1}{2}(f^\ell-f^{-\ell}).
\end{align}

For $x\in\mathbb{F}_{q^2}^\times$, $y\in\mathbb{F}_q$, put
\begin{align}
\mathbf{a}(x)=\left(\begin{array}{cc} x & 0 \\ 0 & x^{-1} \\ \end{array} \right), \qquad
\mathbf{b}(y)=\left(\begin{array}{cc} 1 & y \\ 0 & 1 \\ \end{array} \right).
\end{align}
Let
\begin{align}
\mathbf{e}=\mathbf{a}(1), \qquad \mathbf{a}=\mathbf{a}(e), \qquad \tilde{\mathbf{a}}=\mathbf{a}(\tilde{e}), \qquad
\mathbf{b}_{\varepsilon}=\mathbf{b}(e^{\acute{\varepsilon}}), \quad \varepsilon\in\{+,-\};
\end{align}
also, introduce
\begin{align}
\mathbf{c}=\left(\begin{array}{cc} s_1 & s'_1\tilde{e}^{-1} \\ s'_1\tilde{e} & s_1 \\ \end{array} \right), \qquad
\tilde{\mathbf{c}}=\left(\begin{array}{cc} \tilde{s}_1 & \tilde{s}'_1\tilde{e}^{-1} \\ \tilde{s}'_1\tilde{e} & \tilde{s}_1 \\ \end{array} \right),  \qquad
\mathbf{j}=\left(\begin{array}{cc} 0 & 1 \\ -1 & 0 \\ \end{array} \right).
\end{align}
According to \cite{FH50} Section 5.2, each element of ${\rm SL}(2,q)$ different from $\pm\mathbf{e}$ is conjugate to $\mathbf{a}^{k}$ with $k\not\equiv 0\pmod{2h}$, or $\mu\mathbf{b}_{\varepsilon}$ with $\mu,\varepsilon\in\{+,-\}$, or $\mathbf{c}^{\ell}$ with $\ell\not\equiv0\pmod{2h+1}$.
Furthermore, $\mathbf{a}^{k'}$ is conjugate to $\mathbf{a}^{k}$ if and only if $k'\equiv\pm k\pmod{q-1}$, and $\mathbf{c}^{\ell'}$ is conjugate to $\mathbf{c}^{\ell}$ if and only if $\ell'\equiv \pm\ell\pmod{q+1}$.

%We have
%\begin{align*}
%{\rm Cen}(\mathbf{a}^{k})&\ =\langle\mathbf{a}\rangle\cong\mathbb{Z}_{q-1}, \qquad 1\le k\le 2\overline{q}-1, \\
%{\rm Cen}(\mu\mathbf{b}_{\varepsilon})&\ =\{\pm\mathbf{b}(c)\colon c\in \mathbb{F}_q\}\cong\mathbb{Z}_2\times\mathbb{Z}_{p}^n, \qquad %\mu,\varepsilon\in\{+,-\},  \\
%{\rm Cen}(\mathbf{c}^{\ell})&\ =\langle\mathbf{c}\rangle\cong\mathbb{Z}_{q+1}, \qquad 1\le\ell\le 2\overline{q}.
%\end{align*}

%Let $\tilde{\Gamma}$ denote the subgroup of ${\rm SL}(2,q^2)$ generated by $\tilde{\mathbf{a}}$ and $\Gamma$. Since
%$\tilde{\mathbf{a}}^2\in\Gamma$ and $\tilde{\mathbf{a}}.\mathbf{g}\in\Gamma$ for any $\mathbf{g}\in\Gamma$, we have an extension
%\begin{align*}
%1\to\Gamma\to\tilde{\Gamma}\to\mathbb{Z}_2\to 1.
%\end{align*}
%Note that $\overline{\tilde{f}}=-\tilde{f}^{-1}$, so that both $\tilde{s}'_\ell$ and $\tilde{s}_{\ell}\tilde{e}^{-1}$
%belong to $\mathbb{F}_q$, hence $\tilde{\mathbf{c}}\in\tilde{\Gamma}$.

\subsection{Irreducible representations of ${\rm SL}(2,q)$}

The irreducible representations of ${\rm SL}(2,q)$ were given in \cite{FH50} (Page 71--73); they are (in the notation of \cite{FH50}) $\mathbf{1}$, $V$, $W_{\sigma}$ ($\sigma\in[1,2h-1]$),
$X_{\phi}$ ($\phi\in[1,2h]$), $W',W''$, $X',X''$. Here is the character table:
\begin{center}
\begin{tabular}{|c|c|c|c|c|}
  \hline
  \ & $\mu\mathbf{e}$ & $\mathbf{a}^k$ & $\mu\mathbf{b}_{\varepsilon}$ &  $\mathbf{c}^\ell$ \\
  \hline
  $\mathbf{1}$ & $1$ & $1$ & $1$ & $1$ \\
  \hline
  $V$ & $q$ & $1$ & $0$ & $-1$ \\
  \hline
  $W_{\sigma}$ & $(q+1)\mu^\sigma$ &  $\zeta_{q-1}^{\sigma k}+\zeta_{q-1}^{-\sigma k}$ & $\mu^\sigma$ & 0 \\
  \hline
  $X_{\phi}$ & $(q-1)\mu^\phi$ & 0 & $-\mu^\phi$ & $-\zeta_{q+1}^{\phi\ell}-\zeta_{q+1}^{-\phi\ell}$ \\
  \hline
  $W'$ & $2h+1$ &  $(-1)^k$ & $s_{\varepsilon}$ & 0 \\
  \hline
  $W''$ & $2h+1$ &  $(-1)^k$ & $s_{\overline{\varepsilon}}$ & 0 \\
  \hline
  $X'$ & $2h\mu$ &  0 & $\overline{\mu}s_{\varepsilon}$ & $(-1)^{\ell+1}$ \\
  \hline
  $X''$ & $2h\mu$  & 0 & $\overline{\mu}s_{\overline{\varepsilon}}$ & $(-1)^{\ell+1}$ \\
  \hline
\end{tabular}
\end{center}
where
\begin{align*}
s_{\pm}=\frac{1}{2}(1\pm\sqrt{q}).
\end{align*}

For $u\in\mathbb{Z}_{q-1}$, let $u^\ast$ denote the homomorphism
\begin{align}
\langle\mathbf{a}\rangle\cong\mathbb{Z}_{q-1}\to\mathbb{C}^\times, \qquad \mathbf{a}^{k}\mapsto\zeta_{q-1}^{ku};
\end{align}
for $v\in\mathbb{F}_q$, let $v^\ast$ denote the homomorphism
\begin{align}
\{\mathbf{b}(c)\colon c\in \mathbb{F}_q\}\cong\mathbb{F}_q\to\mathbb{C}^\times, \qquad \mathbf{b}(c)\mapsto\zeta_{p}^{{\rm tr}(vc)};
\end{align}
for $w\in\mathbb{Z}_{q+1}$, let $w^\ast$ denote the homomorphism
\begin{align}
\langle\mathbf{c}\rangle\cong\mathbb{Z}_{q+1}\to\mathbb{C}^\times, \qquad \mathbf{c}^{\ell}\mapsto\zeta_{q+1}^{\ell w}.
\end{align}

As a representation of $\langle\mathbf{a}\rangle\cong\mathbb{Z}_{q-1}$,
\begin{align}
V\cong \oplus^30^\ast\oplus\bigoplus\limits_{2\mid u\ne 0}\oplus^2u^\ast\cong\mathbb{C}\langle\xi^{\mathbf{a}}_{0,1},\xi^{\mathbf{a}}_{0,2},\xi^{\mathbf{a}}_{0,3}\rangle\oplus\bigoplus\limits_{2\mid u\ne 0}\mathbb{C}\langle\xi^{\mathbf{a}}_{u,1},\xi^{\mathbf{a}}_{u,2}\rangle,
\end{align}
where the $\xi^{\mathbf{a}}_{u,j}$ are $\zeta_{q-1}^{u}$-eigenvectors of $\mathbf{a}$.
As a representation of $\{\mathbf{b}(b)\colon b\in\mathbb{F}_q\}\cong\mathbb{F}_{q}$ or of $\langle\mathbf{c}\rangle$, there is a similar decomposition.

Such decompositions for all the irreducible representations are tabulated as below. 
This viewpoint plays a key role in explicit decompositions of tensor products, as done in the final part of the next section.

\begin{center}
\begin{tabular}{|c|c|c|c|}
  \hline
  \ & $\langle\mathbf{a}\rangle$ & $\{\mathbf{b}(b)\colon b\in\mathbb{F}_q\}$ & $\langle\mathbf{c}\rangle$  \\
  \hline
  $V$ & $\oplus^30^\ast\oplus\bigoplus\limits_{2\mid u\ne 0}\oplus^2u^\ast$ & $\bigoplus\limits_{v\in\mathbb{F}_q}v^\ast$ & $0^\ast\oplus\bigoplus\limits_{2\mid w\ne 0}\oplus^2w^\ast$ \\
  \hline
  $W_{\sigma}$ & $\sigma^{\ast}\oplus(-\sigma)^\ast\oplus\bigoplus\limits_{2\mid u-\sigma}\oplus^2u^\ast$ & $2\cdot0^{\ast}\oplus\bigoplus\limits_{v\in\mathbb{F}_q^\times}v^\ast$ &
  $\bigoplus\limits_{2\mid w-\sigma}\oplus^2w^\ast$  \\
  \hline
  $X_{\phi}$ & $\bigoplus\limits_{2\mid u-\phi}\oplus^2u^\ast$ & $\bigoplus\limits_{v\in\mathbb{F}_q}v^\ast$ & $\bigoplus\limits_{2\mid w-\phi, w\ne\pm\phi}\oplus^2w^\ast$  \\
  \hline
  $W'$ & $(2h)^{\ast}\oplus\bigoplus\limits_{2\mid u}u^\ast$ & $0^\ast\oplus\bigoplus\limits_{v\in\mathbb{F}_q^{\times 2}}v^\ast$ & $\bigoplus\limits_{2\mid w}w^\ast$  \\
  \hline
  $W''$ & $(2h)^{\ast}\oplus\bigoplus\limits_{2\mid u}u^\ast$ & $0^\ast\oplus\bigoplus\limits_{v\in\mathbb{F}_q^{\times}-\mathbb{F}_q^{\times 2}}v^\ast$ & $\bigoplus\limits_{2\mid w}w^\ast$ \\
  \hline
  $X'$ & $\bigoplus\limits_{2\nmid u}u^\ast$ & $\bigoplus\limits_{v\in\mathbb{F}_q^{\times 2}}v^\ast$ & $\bigoplus\limits_{2\nmid w,w\ne 2h+1}w^\ast$ \\
  \hline
  $X''$ & $\bigoplus\limits_{2\nmid u}u^\ast$ & $\bigoplus\limits_{v\in\mathbb{F}_q^{\times}-\mathbb{F}_q^{\times 2}}v^\ast$ & $\bigoplus\limits_{2\nmid w, w\ne 2h+1}w^\ast$ \\
  \hline
\end{tabular}
\end{center}

\subsection{Simple objects of $\mathcal{E}({\rm SL}(2,q))$ with nontrivial supports}

For $u\in\mathbb{Z}_{q-1}$, put
\begin{align}
A_k^{u}=U(\mathbf{a}^k,u^\ast).
\end{align}

Note that
\begin{align}
[\mu\mathbf{b}_{\varepsilon}]=
\left\{\mu\left(\begin{array}{cc} 1-e^{\acute{\varepsilon}}ac & e^{\acute{\varepsilon}}a^2
\\ -e^{\acute{\varepsilon}}c^2 & 1+e^{\acute{\varepsilon}}ac \end{array}\right)
\colon (a,c)\in\mathbb{F}_q^2-\{(0,0)\}\right\}.  \label{eq:[b]}
\end{align}
For $\nu\in\{+,-\}$ and $v\in\mathbb{F}_q$, let $(\nu,v)^\ast$ denote the homomorphism
$${\rm Cen}(\mu\mathbf{b}_\varepsilon)\to\mathbb{C}^\times, \qquad $$
and put
\begin{align}
B_{\mu,\varepsilon}^{\nu,v}=U(\mu\mathbf{b}_{\varepsilon},(\nu,v)^\ast).  %\bigoplus\limits_{\eta\in \mathcal{S}(\beta)}\mathbb{C}\langle \eta.\beta(c)\rangle.
\end{align}

For $w\in\mathbb{Z}_{q+1}$, put
\begin{align}
C_{\ell}^w=U(\mathbf{c}^\ell,w^\ast).
\end{align}

Summarized as follows:
\begin{center}
\begin{tabular}{|c|c|c|c|c|}
  \hline
  $U$ & ${\rm supp}(U)$ & ranges\ of\ parameters & $\theta_U$ & $d_U$ \\
  \hline
  $A_k^u$ & $[\mathbf{a}^k]$ & $k\in[1,2h-1], u\in[1,q-1]$  & $\zeta_{q-1}^{ku}$ & $q(q+1)$ \\
  \hline
  $B_{\mu,\varepsilon}^{\nu,v}$ & $[\mu\mathbf{b}_\varepsilon]$ & $\varepsilon,\mu,\nu\in\{+,-\}, v\in\mathbb{F}_q$
  & $\mu\nu\zeta_p^{{\rm tr}(e^{\acute{\varepsilon}}v)}$ & $(q^2-1)/2$ \\
  \hline
  $C_\ell^w$ & $[\mathbf{c}^\ell]$ & $\ell\in[1,2h], w\in[1,q+1]$ & $\zeta_{q+1}^{\ell w}$ & $q(q-1)$ \\
  \hline
\end{tabular}
\end{center}

Each simple object is self-dual.
Indeed, we have isomorphisms
\begin{align*}
&(A_k^{u})_{\mathbf{a}^k}\cong((A_k^{u})^\ast)_{\mathbf{a}^k}=((A_k^{u})_{\mathbf{a}^{-k}})^\ast, && \mathbf{a}^k\mapsto (\mathbf{a}^{-k}\mapsto 1), \\
&(B_{\mu,\varepsilon}^{\nu,v})_{\mu\mathbf{b}_{\varepsilon}}\cong ((B_{\mu,\varepsilon}^{\nu,v})^\ast)_{\mu\mathbf{b}_{\varepsilon}}=((B_{\mu,\varepsilon}^{\nu,v})_{\mu\mathbf{b}_{\varepsilon}})^\ast,
&&\mu\mathbf{b}_{\varepsilon}\mapsto ((\mu\mathbf{b}_{\varepsilon})^{-1}\mapsto 1), \\
&(C_\ell^{w})_{\mathbf{c}^\ell}\cong((C_\ell^{w})^\ast)_{\mathbf{c}^\ell}=((C_\ell^{w})_{\mathbf{c}^{-\ell}})^\ast, && \mathbf{c}^\ell\mapsto (\mathbf{c}^{-\ell}\mapsto 1).
\end{align*}

For $X=U(\mathbf{x},\rho)$, if $\rho(-\mathbf{e})=1$ (resp. $\rho(-\mathbf{e})=-1$), then we say that the {\it parity} of $X$ is even (resp. odd), and write $\nu(X)=+$ (resp. $\nu(X)=-$).

\section{Tensor products and braidings}

Let
\begin{align}
\mathbf{k}=\left(\begin{array}{cc} 1 & -\tilde{e}^{-1} \\ \tilde{e} & 1 \end{array}\right).
\end{align}
Note that $\mathbf{c}=\mathbf{k}.\mathbf{a}(f)$, so $\mathbf{c}^\ell=\mathbf{k}.\mathbf{a}(f^\ell)$.

\begin{lem} \label{lem:real}
For $\mathbf{u}=\left(\begin{array}{cc} a & b \\ c & d  \end{array}\right)\in{\rm SL}(2,q^2)$, $\mathbf{k}.\mathbf{u}\in{\rm SL}(2,q)$
if and only if $d=\overline{a}$, $c=-e\overline{b}$ (so that $\mathcal{N}(a)+e\mathcal{N}(b)=1$).
\end{lem}

\begin{proof}
Computing directly,
\begin{align}
\mathbf{k}.\mathbf{u}=\frac{1}{2e}\left(\begin{array}{cc} (a+d)e-(be+c)\tilde{e} & (a-d)\tilde{e}+(be-c) \\
(a-d)e\tilde{e}-(be-c)e & (a+d)e+(be+c)\tilde{e} \end{array}\right).  \label{eq:k.u}
\end{align}
Clearly, $\mathbf{k}.\mathbf{u}\in{\rm SL}(2,q)$ if and only if $a+d$, $be-c$, $(a-d)\tilde{e}$ and $b\tilde{e}+c\tilde{e}^{-1}$
are all in $\mathbb{F}_q$, which is equivalent to $d=\overline{a}$, $c=-e\overline{b}$.
\end{proof}

\begin{lem} \label{lem:multiply}
Let $\mathbf{y}_1,\mathbf{y}_2\in{\rm SL}(2,q^2)$ with ${\rm tr}(\mathbf{y}_j)=t_j, j=1,2$.
\begin{enumerate}
  \item[\rm(i)] For any $a\in\mathbb{F}_{q^2}^{\times}-\{\pm1\}$, $\mathbf{y}_1\mathbf{y}_2=\mathbf{a}(a)$ if and only if
       \begin{align*}
       \mathbf{y}_1&=\frac{1}{a-a^{-1}}\left(\begin{array}{cc} at_1-t_2 & -ay \\ -a^{-1}z & t_2-a^{-1}t_1 \end{array}\right),  \\
       \mathbf{y}_2&=\frac{1}{a-a^{-1}}\left(\begin{array}{cc} at_2-t_1 & y \\ z & t_1-a^{-1}t_2 \end{array}\right)
       \end{align*}
       for some $b, c$ with $bc=(a+a^{-1})t_1t_2-t_1^2-t_2^2-(a-a^{-1})^2$.
  \item[\rm(ii)] For $\mu\in\{\pm 1\}$ and $b\in\mathbb{F}_q^\times$, $\mathbf{y}_1\mathbf{y}_2=\mu\mathbf{b}(b)$ if and only if
       \begin{align*}
       \mathbf{y}_1&=\left(\begin{array}{cc} t_1-\mu x & \mu(bx-y) \\ (t_1-\mu t_2)b^{-1} & \mu x \end{array}\right),  \\
       \mathbf{y}_2&=\left(\begin{array}{cc} x & y \\ (t_2-\mu t_1)b^{-1} & t_2-x \end{array}\right)
       \end{align*}
       for some $x,y$ with $x(t_2-x)=(t_2-\mu t_1)yb^{-1}+1$.
\end{enumerate}
\end{lem}

\begin{proof}
Suppose
$\mathbf{y}_j=\left(\begin{array}{cc} x_j & y_j \\ z_j & w_j \end{array}\right)$.
Then
\begin{enumerate}
  \item[\rm(i)] $\mathbf{y}_1=\mathbf{a}(a)\mathbf{y}_2^{-1}$ reads
       $$x_1=aw_2, \qquad y_1=-ay_2, \qquad z_1=-a^{-1}z_2, \qquad w_1=a^{-1}x_2,$$
       which together with ${\rm tr}(\mathbf{y}_j)=t_j$ implies the result.
  \item[\rm(ii)] $\mathbf{y}_1=\mathbf{b}(b)\mathbf{y}_2^{-1}$ reads
       $$\mu x_1=w_2-bz_2, \qquad \mu y_1=bx_2-y_2, \qquad \mu z_1=-z_2, \qquad \mu w_1=x_2,$$
       which together with ${\rm tr}(\mathbf{y}_j)=t_j$ implies the result.
\end{enumerate}
\end{proof}

\begin{rmk}
\rm Note that in (i), denoting $t=a+a^{-1}$ and $t_j=a_j+a_j^{-1}$,
\begin{align*}
bc&=-(t-a_1a_2-a_1^{-1}a_2^{-1})(t-a_1a_2^{-1}-a_1^{-1}a_2) \\
&=4+tt_1t_2-t^2-t_1^2-t_2^2=:\Delta(t,t_1,t_2),
\end{align*}
hence $bc=0$ if and only if $a=a_1^{\epsilon_1}a_2^{\epsilon_2}$ with $\epsilon_j\in\{\pm 1\}$.
\end{rmk}

Given two simple objects $X_1,X_2$, consider decomposing $X_1\odot X_2$ into a direct sum of simple objects.
Suppose ${\rm supp}(X_j)=[\mathbf{x}_j]$, $j=1,2$.

If $\mathbf{x}_1=\mu\mathbf{e}$, then
$$X_1\odot X_2\cong\bigoplus\limits_{\rho\in\mathcal{R}_{\mathbf{x}_2}}\oplus^{n_\rho}U(\mathbf{x}_2,\rho).$$
We may fix an isomorphism, and define
$$X_2\odot X_1\cong\bigoplus\limits_{\rho\in\mathcal{R}_{\mathbf{x}_2}}\oplus^{n_\rho}U(\mathbf{x}_2,\rho)$$
by post-composing the switching map $X_2\odot X_1\to X_1\odot X_2$.
Then
\begin{align*}
R_{X_1,X_2}=\bigoplus\limits_{\rho\in\mathcal{R}_{\mathbf{x}_2}}I(n_\rho),  \qquad
R_{X_2,X_1}=\bigoplus\limits_{\rho\in\mathcal{R}_{\mathbf{x}_2}}\theta_{U(\mathbf{x}_2,\rho)}\cdot I(n_\rho).
\end{align*}
So it suffices to determined the multiplicities $n_\rho$. We omit the computations.
%From the perspective of computing DW invariant, the component colored by $U$ can be omitted.

Suppose $\mathbf{x}_j\ne\pm\mathbf{e}$, $j=1,2$. Let $t_j=a_j+a_j^{-1}$, with
\begin{align}
a_j=\begin{cases}
e^{k'}, &\mathbf{x}_j=\mathbf{a}^{k'}, \  \ k'\in[1,2h-1], \\
\mu, &\mathbf{x}_j=\mu\mathbf{b}_{\varepsilon},  \  \mu,\varepsilon\in\{+,-\},  \\
f^{\ell'}, &\mathbf{x}_j=\mathbf{c}^{\ell'},  \ \ \ell'\in[1,2h].
\end{cases}
\end{align}
Let $\nu=+$ (resp. $\nu=-$) if the parities of $X_1,X_2$ are the same (resp. opposite).

According to supports, we can primarily decompose $X=X_1\odot X_2$ as
\begin{align}
X\cong\bigoplus\limits_{k=1}^{2h-1}X_{[\mathbf{a}^k]}\oplus \bigoplus\limits_{\mu,\varepsilon\in\{+,-\}}X_{[\mu\mathbf{b}_\varepsilon]}
\oplus\bigoplus\limits_{\ell=1}^{2h}X_{[\mathbf{c}^\ell]}\oplus X_{[\mathbf{e}]}\oplus X_{[-\mathbf{e}]}.
\end{align}
Denote the restriction of $R$ to $X_{[\mathbf{x}]}$ by $R_{[\mathbf{x}]}$.

\subsection{$X_{[\mathbf{a}^k]}$, $1\le k\le 2h-1$}

Let $\mathbb{G}$ denote the set of $u\in[1,q-1]$ with $u\equiv\acute{\nu}\pmod{2}$.

\subsubsection{$e^k\ne a_1^{\epsilon_1}a_2^{\epsilon_2}$}

For $i\in[1,,q-1]$, put
\begin{align}
\mathbf{f}_j(i)&=\frac{1}{e^k-e^{-k}}\left(\begin{array}{cc} e^kt_j-t_{\overline{j}} & e^{i} \\
\star & t_{\overline{j}}-e^{-k}t_j \end{array}\right), \qquad j=1,2,
\end{align}
where $\star\ne 0$ is determined by the condition $\det(\mathbf{f}_j(i))=1$; let
\begin{align}
\mathbf{f}(i)&=\mathbf{f}_1(2h+k+i)\otimes\mathbf{f}_2(i).
\end{align}
Suppose $\tilde{\mathbf{a}}^i:\mathbb{C}\langle\mathbf{f}_j(0)\rangle\to\mathbb{C}\langle\mathbf{f}_j(i)\rangle$ is given by
$\mathbf{f}_j(0)\mapsto\eta_j^\mathbf{a}(i)\mathbf{f}_j(i)$.

If $\mathbf{x}_j\ne\mu\mathbf{b}_\varepsilon$, $j=1,2$, then $X_{\mathbf{a}^k}=X^{\rm even}_{\mathbf{a}^k}\oplus X^{\rm odd}_{\mathbf{a}^k}$, with
\begin{align*}
X^{\rm even}_{\mathbf{a}^k}=\bigoplus\limits_{i'=1}^{2h}\mathbb{C}\langle\mathbf{f}(2i')\rangle, \qquad
X^{\rm odd}_{\mathbf{a}^k}=\bigoplus\limits_{i'=1}^{2h}\mathbb{C}\langle\mathbf{f}(2i'-1)\rangle.
\end{align*}
Clearly, $X^{\rm even}_{\mathbf{a}^k}$ is a representation of $\langle\mathbf{a}\rangle$, and the character vanishes at $\mathbf{a}^{k'}$ for all $k'$; similarly for $X^{\rm odd}_{\mathbf{a}^k}$. Hence, regarding of the action of $-\mathbf{e}$,
\begin{align}
X_{[\mathbf{a}^k]}\cong\bigoplus\limits_{u\in\mathbb{G}}\oplus^2A_k^{u};
\end{align}
the isomorphism can be given by the one determined by
\begin{align}
\mathbf{f}(i)\mapsto\bigoplus\limits_{u\in\mathbb{G}}\vartheta_u(i)\mathbf{z}(i),  \label{eq:iso-A1}
\end{align}
with
\begin{align}
\vartheta_u(i)=\frac{\zeta_{2(q-1)}^{iu}}{\eta^{\mathbf{a}}_1(2h+k+i)\eta^{\mathbf{a}}_2(i)}, \qquad
\mathbf{z}(i)=\begin{cases} (1,0),&2\mid i, \\ (0,1),&2\nmid i. \end{cases}
\end{align}
To compute $R_{[\mathbf{a}^k]}(\mathbf{f}(i))$, we need to know the action of $\mathbf{f}_1(2h+k+i)$ on $\mathbf{f}_2(i)$.
Since $\mathbf{f}_1(2h+k+i)=\mathbf{a}^k\mathbf{f}_2(i)^{-1}$, we have
\begin{align*}
R_{[\mathbf{a}^k]}(\mathbf{f}(i))=
\frac{\eta^{\mathbf{a}}_2(2k+i)}{\eta^{\mathbf{a}}_2(i)\theta_{X_2}}\cdot\mathbf{f}_2(2k+i)\otimes\mathbf{f}_1(2h+k+i)
\end{align*}
which is sent by (\ref{eq:iso-A1}) to
$$\bigoplus\limits_{u\in\mathbb{G}}\theta_{X_2}^{-1}\cdot\zeta_{2(q-1)}^{(k+2h)u}\vartheta_u(i)\cdot\mathbf{z}(k+i).$$
It follows that
\begin{align}
R_{[\mathbf{a}^k]}=\theta_{X_2}^{-1}\cdot\bigoplus\limits_{u\in\mathbb{G}}\zeta_{2(q-1)}^{(k+2h)u}J^{k},
\qquad \text{with} \qquad
J=\left(\begin{array}{cc} 0 & 1 \\ 1 & 0 \\ \end{array}\right).
\end{align}

\begin{rmk}
\rm Note that the result is not ``well-defined" with respect to the choice of $u$. This is due to our definition of
$\eta^{\mathbf{a}}_j(i)$, which could have been defined differently according to the parity of $i$.
\end{rmk}

If $\mathbf{x}_j=\mu_j\mathbf{b}_{\varepsilon_j}$ for at least one $j$, then
\begin{align}
X_{\mathbf{a}^k}=\bigoplus\limits_{i\in\mathbb{H}}\mathbb{C}\langle\mathbf{f}(i)\rangle,
\end{align}
where $\mathbb{H}$ is the set of $i\in[1,q-1]$ with (denoting $e^k-e^{-k}=e^{\check{k}}$)
$$i+\check{k}\equiv jk+\acute{\varepsilon_j}\pmod{2};$$
note that if $\mathbf{x}_1=\mu_1\mathbf{b}_{\varepsilon_1}$ and $\mathbf{x}_2=\mu_2\mathbf{b}_{\varepsilon_2}$, then $k\equiv \acute{\varepsilon_1}-\acute{\varepsilon_2}\pmod{2}$, so that $\mathbb{H}$ is still nonempty.
Hence
\begin{align}
X_{[\mathbf{a}^k]}\cong\bigoplus\limits_{u\in\mathbb{G}}A_k^{u},
\end{align}
the isomorphism determined by
\begin{align}
\mathbf{f}(i)\mapsto\bigoplus\limits_{u\in\mathbb{G}}\frac{\zeta_{2(q-1)}^{iu}}{\eta^{\mathbf{a}}_{1}(2h+k+i)\eta^{\mathbf{a}}_{2}(i)}.
\end{align}
Similarly as the previous part,
\begin{align}
R_{[\mathbf{a}^k]}=\theta_{X_2}^{-1}\cdot\bigoplus\limits_{u\in\mathbb{G}}\zeta_{2(q-1)}^{(k+2h)u}I.
\end{align}

\subsubsection{$e^k=a_1^{\epsilon_1}a_2^{\epsilon_2}$ with $\epsilon_j\in\{\pm 1\}$}

In this case, $a_j=e^{k_j}, j=1,2$, with $\epsilon_1k_1+\epsilon_2k_2=k\pmod{q-1}$. It is possible that
$k_j\in\{0,2h\}$ for some $j$, in which case $\mathbf{x}_j=\mu_j\mathbf{b}_{\varepsilon_j}$.

For $i\in[1,q-1]$, put
\begin{align}
\mathbf{f}^+_j(i)&=\left(\begin{array}{cc} e^{\epsilon_jk_j} & e^i \\ 0 & e^{-\epsilon_jk_j} \end{array}\right), \quad
\mathbf{f}^-_j(i)=\left(\begin{array}{cc} e^{\epsilon_jk_j} & 0 \\ e^{-i} & e^{-\epsilon_jk_j} \end{array}\right), \quad j=1,2, \\
\mathbf{f}^\psi(i)&=\mathbf{f}^\psi_{1}(2h+k+i)\otimes\mathbf{f}^\psi_{2}(i), \qquad \psi\in\{+,-\}.
\end{align}
Suppose $\tilde{\mathbf{a}}^i:\mathbb{C}\langle\mathbf{f}^\psi_j(0)\rangle\to\mathbb{C}\langle\mathbf{f}^\psi_j(i)\rangle$ is given by
$\mathbf{f}^\psi_j(0)\mapsto\eta^{\mathbf{a},\psi}_{j}(i)\cdot\mathbf{f}^\psi_j(i)$.

If $X_1=A_{k_1}^{u_1}$, $X_2=A_{k_2}^{u_2}$, then
\begin{align}
X_{\mathbf{a}^k}&=\mathbb{C}\langle\mathbf{a}^{\epsilon_1k_1}\otimes\mathbf{a}^{\epsilon_2k_2}\rangle\oplus
\bigoplus\limits_{i=1}^{q-1}\mathbb{C}\langle\mathbf{f}^+(i),\mathbf{f}^-(i)\rangle, \\
X_{[\mathbf{a}^k]}&\cong A_k^{\epsilon_1u_1+\epsilon_2u_2}\oplus\bigoplus\limits_{u\in\mathbb{G}}\oplus^4A_k^{u},
\end{align}
the isomorphism determined by
\begin{align}
\mathbf{a}^{\epsilon_1k_1}\otimes\mathbf{a}^{\epsilon_2k_2}&\mapsto \mathbf{a}^k\in A_k^{\epsilon_1u_1+\epsilon_2u_2}, \\
\mathbf{f}^+(i)&\mapsto\bigoplus\limits_{u\in\mathbb{G}}
\frac{\zeta_{2(q-1)}^{iu}}{\eta^{\mathbf{a},+}_{1}(2h+k+i)\eta^{\mathbf{a},+}_{2}(i)}(\mathbf{z}(i),\mathbf{0}), \\
\mathbf{f}^-(i)&\mapsto\bigoplus\limits_{u\in\mathbb{G}}
\frac{\zeta_{2(q-1)}^{iu}}{\eta^{\mathbf{a},-}_{1}(2h+k+i)\eta^{\mathbf{a},-}_{2}(i)}(\mathbf{0},\mathbf{z}(i)).
\end{align}
The braiding is
\begin{align}
R_{[\mathbf{a}^k]}=\zeta_{2h}^{\epsilon_1k_1u_2}\oplus
\theta_{X_2}^{-1}\cdot\bigoplus\limits_{u\in\mathbb{G}}\zeta_{2(q-1)}^{(k+2h)u}\oplus^2J^{k}.
\end{align}

If $\mathbf{x}_j=\mu\mathbf{b}_\varepsilon$, then $\mathbf{x}_{\overline{j}}=\mathbf{a}^{k_\mu}$, and
\begin{align}
X_{\mathbf{a}^k}&=\bigoplus\limits_{i\in\mathbb{H}}\mathbb{C}\langle\mathbf{f}^+(i),\mathbf{f}^-(i)\rangle,
\end{align}
Thus
\begin{align}
X_{[\mathbf{a}^k]}&\cong \bigoplus\limits_{u\in\mathbb{G}}\oplus^2A_k^{u},
\end{align}
the isomorphism determined by
\begin{align}
\mathbf{f}^\psi(i)\mapsto\bigoplus\limits_{u\in\mathbb{G}}
\frac{\zeta_{2(q-1)}^{iu}}{\eta^{\mathbf{a},\psi}_1(2h+k+i)\eta^{\mathbf{a},\psi}_{2}(i)}\cdot\mathbf{z}(\psi).
\end{align}
The braiding is
\begin{align}
R_{[\mathbf{a}^k]}=\theta_{X_2}^{-1}\cdot\bigoplus\limits_{u\in\mathbb{G}}\zeta_{2(q-1)}^{(k+2h)u}I.
\end{align}

\subsection{$X_{[\mu\mathbf{b}_{\varepsilon}]}$, $\mu,\varepsilon\in\{+,-\}$}

\subsubsection{$t_1\ne \mu t_2$}

Let
\begin{align}
\mathbf{g}_j&=\left(\begin{array}{cc}  0 & e^{\acute{\varepsilon}}/(\mu t_{\overline{j}}-t_j) \\
(t_j-\mu t_{\overline{j}})e^{-\acute{\varepsilon}} & t_j \\ \end{array}\right), \qquad j=1,2,  \\
\mathbf{g}(x)&=\mathbf{b}(x+c).\mathbf{g}_1\otimes\mathbf{b}(x).\mathbf{g}_2, \qquad \text{with} \qquad
c=\frac{e^{\acute{\varepsilon}}t_1}{t_1-\mu t_2}.
\end{align}
By (\ref{eq:[b]}), $X_{\mu\mathbf{b}_{\varepsilon}}\ne 0$ requires
\begin{align}
(\mu t_{\overline{j}}-t_j)e^{\acute{\varepsilon}+\acute{\varepsilon_j}}\in\mathbb{F}_q^{\times2} \qquad \text{if} \qquad X_j=B_{\mu_j,\varepsilon_j}^{\nu_j,v_j}.
\end{align}
When this holds, by Lemma \ref{lem:multiply} (ii),
\begin{align}
X_{\mu\mathbf{b}_{\varepsilon}}=\bigoplus\limits_{x\in\mathbb{F}_q}\mathbf{g}(x).
\end{align}
Suppose $\mathbf{b}(x):\mathbb{C}\langle\mathbf{g}_j\rangle\to\mathbb{C}\langle\mathbf{b}(x).\mathbf{g}_j\rangle$ is given by $\mathbf{g}_j\mapsto\eta^{\mathbf{b}}_{j}(x)\cdot\mathbf{b}(x).\mathbf{g}_j$.

We have an isomorphism
\begin{align}
X_{[\mu\mathbf{b}_{\varepsilon}]}&\cong\bigoplus\limits_{v\in\mathbb{F}_q}B_{\mu,\varepsilon}^{\nu,v}, \\
\mathbf{g}(x)&\mapsto\bigoplus\limits_{v\in\mathbb{F}_q}
\frac{\zeta_p^{{\rm tr}(vx)}}{\eta^{\mathbf{b}}_1(x+c)\eta^{\mathbf{b}}_2(x)}.  \label{eq:iso-B1}
\end{align}

To compute the braiding, note that $\mathbf{b}(x+c).\mathbf{g}_1=\mu\mathbf{b}_{\varepsilon}(\mathbf{b}(x).\mathbf{g}_2)^{-1}$, so that
$$R_{[\mu\mathbf{b}_\varepsilon]}(\mathbf{g}(x))=
\frac{\mu^{\acute{\nu_2}}}{\theta_{X_2}}\frac{\eta^{\mathbf{b}}_2(e^{\acute{\varepsilon}}+x)}{\eta^{\mathbf{b}}_2(x)}
\mathbf{b}(e^{\acute{\varepsilon}}+x).\mathbf{g}_2\otimes\mathbf{b}(x+c).\mathbf{g}_1,$$
which is sent by (\ref{eq:iso-B1}) to
$$\frac{\mu^{\acute{\nu_2}}}{\theta_{X_2}}\cdot\bigoplus\limits_{v\in\mathbb{F}_q}
\frac{\zeta_p^{{\rm tr}(v(x-c))}}{\eta_1^{\mathbf{b}}(x+c)\eta_2^{\mathbf{b}}(x)}.$$
Hence
\begin{align}
R_{[\mu\mathbf{b}_\varepsilon]}=\frac{\mu^{\acute{\nu_2}}}{\theta_{X_2}}\bigoplus\limits_{v\in\mathbb{F}_q}
\zeta_p^{{\rm tr}(\frac{e^{\acute{\varepsilon}}t_1v}{t_1-\mu t_2})}.
\end{align}

\subsubsection{$t_1=\mu t_2\ne\pm 2$}

By Lemma \ref{lem:multiply} (ii), $\mathbf{y}_1,\mathbf{y}_2$ are upper-triangular, so that $a_j=e^{k_j}$ and $X_j=A_k^{u_j}$, $j=1,2$, with $k_2=k_1(\mu)$.
For $\psi\in\{+,-\}$, put
\begin{align}
\mathbf{g}^{\psi}(x)=\mathbf{b}(x+c').\mathbf{a}^{\overline{\psi}k_1}\otimes\mathbf{b}(x).\mathbf{a}^{\psi k_2}, \qquad
\text{with} \qquad
c'=\frac{e^{\acute{\varepsilon}+k_2}}{e^{k_2}-e^{-k_2}}.
\end{align}
Then
\begin{align}
X_{\mu\mathbf{b}_\varepsilon}=\bigoplus\limits_{x\in\mathbb{F}_q}\mathbb{C}\langle\mathbf{g}^+(x)\mathbf{g}^-(x)\rangle,
\end{align}
Suppose $\mathbf{b}(x):\mathbb{C}\langle\mathbf{a}^{\psi k_j}\rangle\to\mathbb{C}\langle\mathbf{b}(x).\mathbf{a}^{\psi k_j}\rangle$
is given by $\mathbf{a}^{\psi k_j}\mapsto\eta^{\mathbf{b},\psi}_j(x)\cdot\mathbf{b}(x).\mathbf{a}^{\psi k_j}$.

We have
\begin{align}
X_{[\mu\mathbf{b}_\varepsilon]}\cong\bigoplus\limits_{v\in\mathbb{F}_q}\oplus^2B_{\mu,\varepsilon}^{\nu,v}, \qquad
\mathbf{g}^{\psi}(x)\mapsto\bigoplus\limits_{v\in\mathbb{F}_q}\frac{\zeta_p^{{\rm tr}(vx)}}{\eta^{\mathbf{b},\overline{\psi}}_1(x+c')\eta^{\mathbf{b},\psi}_2(x)}\mathbf{z}(\psi). \label{eq:iso-B2}
\end{align}
Since $\mathbf{b}(x+c').\mathbf{a}^{\overline{\psi}k_1}=\mu\mathbf{b}_\varepsilon(\mathbf{b}(x).\mathbf{a}^{\psi k_2})^{-1}$,
\begin{align*}
R_{[\mu\mathbf{b}_\varepsilon]}(\mathbf{g}^\psi(x))&=
\frac{\mu^{\acute{\nu_2}}}{\theta_{X_2}}\frac{\eta_{2}^{\mathbf{b},\psi}(e^{\acute{\varepsilon}}+x)}{\eta_{2}^{\mathbf{b},\psi}(x)}\cdot
\mathbf{b}(e^{\acute{\varepsilon}}+x).a^{\psi k_2}\otimes\mathbf{b}(x+c').\mathbf{a}^{\overline{\psi} k_1},
\end{align*}
which is sent by (\ref{eq:iso-B2}) to
$$\frac{\mu^{\acute{\nu_2}}}{\theta_{X_2}}\cdot\bigoplus\limits_{v\in\mathbb{F}_q}
\frac{\zeta_p^{{\rm tr}(v(x+c'))}}{\eta^{\mathbf{b},\overline{\psi}}_{1}(x+c')\eta^{\mathbf{b},\psi}_{2}(x)}\cdot\mathbf{z}(\overline{\psi}).$$
Hence the braiding is
\begin{align}
R_{[\mu\mathbf{b}_\varepsilon]}=\bigoplus\limits_{v\in\mathbb{F}_q}\zeta_p^{{\rm tr}(\frac{e^{\acute{\varepsilon}+k_2}v}{e^{k_2}-e^{-k_2}})}J.
\end{align}

\subsubsection{$t_1=\mu t_2\in\{\pm 2\}$}

Now $t_j=2\mu_j$ and $X_j=B_{\mu_j,\varepsilon_j}^{\nu_j,v_j}$ for $j=1,2$, with $\mu_1\mu_2=\mu$, $\varepsilon_1\varepsilon_2=\varepsilon$.
By (\ref{eq:[b]}) and Lemma \ref{lem:multiply} (ii),
\begin{align*}
X_{\mu\mathbf{b}_\varepsilon}=\bigoplus\limits_{i\in\mathbb{H}'}
\mathbb{C}\langle\mu_1\mathbf{b}(e^{\acute{\varepsilon}}-e^{2i+\acute{\varepsilon_2}})\rangle\otimes
\mathbb{C}\langle\mu_2\mathbf{b}(e^{2i+\acute{\varepsilon_2}})\rangle,
\end{align*}
where $\mathbb{H}'$ consists of $i\in\{1,\ldots,2\overline{q}\}$ with $(e^{\acute{\varepsilon}}-e^{2i+\acute{\varepsilon_2}})e^{-\acute{\varepsilon_1}}\in\mathbb{F}_q^{\times2}$;
for each $i\in\mathbb{H}'$, write $e^{\acute{\varepsilon}}-e^{2i+\acute{\varepsilon_2}}=e^{2\hat{i}+\acute{\varepsilon_1}}$.

Then
\begin{align}
&X_{[\mu\mathbf{b}_\varepsilon]}\cong\bigoplus\limits_{d\in\mathbb{H}'}B_{\mu,\varepsilon}^{\nu,e^{-2\hat{d}}v_1+e^{-2d}v_2}, \\
&\mu_1\mathbf{b}(e^{2\hat{i}+\acute{\varepsilon_1}})\otimes\mu_2\mathbf{b}(e^{2i+\acute{\varepsilon_2}})
\mapsto \bigoplus\limits_{d\in\mathbb{H}'}\delta_{d,i}. \label{eq:iso-B3}
\end{align}
For the braiding, note that $\mu_1\mathbf{b}(e^{2\hat{i}+\acute{\varepsilon_1}})=(\mu\mathbf{b}_\varepsilon)(\mu_2\mathbf{b}(e^{2i+\acute{\varepsilon_2}}))^{-1}$, so
\begin{align*}
R_{[\mu\mathbf{b}_\varepsilon]}(\mu_1\mathbf{b}(e^{2\hat{i}+\acute{\varepsilon_1}})\otimes\mu_2\mathbf{b}(e^{2i+\acute{\varepsilon_2}}))
&=\frac{\mu^{\acute{\nu_2}}\zeta_p^{{\rm tr}(e^{\acute{\varepsilon}-2i}v_2)}}{\theta_{X_2}}
\mu_2\mathbf{b}(e^{2i+\acute{\varepsilon_2}})\otimes\mu_1\mathbf{b}(e^{2\hat{i}+\acute{\varepsilon_1}}) \\
&\stackrel{(\ref{eq:iso-B3})}\mapsto
\frac{\mu^{\acute{\nu_2}}}{\theta_{X_2}}\zeta_p^{{\rm tr}(e^{\acute{\varepsilon}-2i}v_2)}\bigoplus\limits_{d\in\mathbb{H}'}\delta_{d,\hat{i}}.
\end{align*}
Thus
\begin{align}
&R_{[\mu\mathbf{b}_\varepsilon]}=\bigoplus\limits_{d\in\mathbb{H}'}
\frac{\mu^{\acute{\nu_2}}}{\theta_{X_2}}\zeta_p^{{\rm tr}(e^{\acute{\varepsilon}-2d}v_2)}J(d,\hat{d}), \\
\text{with} \qquad &J(d,\hat{d}):B_{\mu,\varepsilon}^{\nu,e^{-2\hat{d}}v_1+e^{-2d}v_2}\to B_{\mu,\varepsilon}^{\nu,e^{-2d}v_2+e^{-2\hat{d}}v_1}.
\end{align}

\subsection{$X_{[\mathbf{c}^\ell]}$}

Let $\mathbb{I}$ denote the set of $i\in[1,q+1]$ with $i\equiv \acute{\nu}\pmod{2}$.

\subsubsection{$f^\ell\ne a_1^{\epsilon_1}a_2^{\epsilon_2}$}

Fix $x_0\in\mathbb{F}_{q^2}$ with $\mathcal{N}(x_0)=\Delta(t,t_1,t_2)e^{-1}$.

For $i\in\{1,\ldots,q+1\}$, put
\begin{align}
\mathbf{h}_j(i)=\mathbf{k}.\left(\frac{1}{f^\ell-f^{-\ell}}\left(\begin{array}{cc} f^\ell t_j-t_{\overline{j}} & x_0f^i \\ e\overline{x_0}f^{-i} & t_{\overline{j}}-f^{-\ell}t_j \end{array}\right)\right), \qquad j=1,2,
\end{align}
and let
\begin{align}
\mathbf{h}(i)=\mathbf{h}_1(2h+1+\ell+i)\otimes\mathbf{h}_2(i).
\end{align}
Suppose $\tilde{\mathbf{c}}^i:\mathbb{C}\langle\mathbf{h}_j(0)\rangle\to\mathbb{C}\langle\mathbf{h}_j(i)\rangle$ is given by $\mathbf{h}_j(0)\mapsto\eta^{\mathbf{c}}_{j}(i)\cdot\mathbf{h}_j(i)$.

If $\mathbf{x}_j\ne\mu\mathbf{b}_{\varepsilon}$, $j=1,2$, then by Lemma \ref{lem:multiply} (i) and Lemma \ref{lem:real},
\begin{align*}
X_{\mathbf{c}^\ell}=\bigoplus\limits_{i=1}^{q+1}\mathbb{C}\langle\mathbf{h}(i)\rangle,
\end{align*}
hence
\begin{align}
X_{[\mathbf{c}^\ell]}\cong\bigoplus\limits_{w\in\mathbb{I}}\oplus^2C_{\ell}^w,  \qquad \mathbf{h}(i)\mapsto\bigoplus\limits_{w\in\mathbb{I}}\vartheta_w(i)\mathbf{z}(i),
\end{align}
with
\begin{align}
\vartheta_w(i)=\frac{\zeta_{2(q+1)}^{iw}}{\eta^{\mathbf{c}}_{1}(2h+1+\ell+i)\eta^{\mathbf{c}}_{2}(i)}.
\end{align}
The braiding is
\begin{align}
R_{[\mathbf{c}^\ell]}=\theta_{X_2}^{-1}\cdot\bigoplus\limits_{w\in\mathbb{I}}\zeta_{2(q+1)}^{(\ell+2h+1)w}J^{\ell+1}.
\end{align}

If $X_j=B_{\mu_j,\varepsilon_j}^{\nu_j,v_j}$ for at least one $j$, then $x_0$ can chose to be $(t-\mu_jt_{\overline{}j})\tilde{e}^{-1}$. By (\ref{eq:k.u}), $\mathbf{h}_j(i)\in[\mu_j\mathbf{b}_{\varepsilon_j}]$ if and only if
$$\frac{(t-\mu_jt_{\overline{j}})e^{\acute{\varepsilon_j}}}{2\tilde{e}(f^\ell-f^{-\ell})}(\tilde{f}^{i}+\tilde{f}^{-i})^2\in\mathbf{F}_q^{\times2};$$
let $\mathbb{K}$ denote the set of $i\in[1,q+1]$ determined by $\mathbf{h}_1(2h+1+\ell+i)\in[\mu_1\mathbf{b}_{\varepsilon_1}]$ (resp. $\mathbf{h}_2(i)\in[\mu_2\mathbf{b}_{\varepsilon_2}]$) if $j=1$ (resp. $j=2$). Then
\begin{align*}
X_{\mathbf{c}^\ell}=\bigoplus\limits_{i\in\mathbb{K}}\mathbb{C}\langle\mathbf{h}(i)\rangle,
\end{align*}
hence
\begin{align}
X_{[\mathbf{c}^\ell]}\cong\bigoplus\limits_{w\in\mathbb{I}}C_{\ell}^{w}, \qquad \mathbf{h}(i)\mapsto\bigoplus\limits_{w\in\mathbb{I}}
\frac{\zeta_{2(q+1)}^{iw}}{\eta^{\mathbf{c}}_{1}(2h+1+\ell+i)\eta^{\mathbf{c}}_{2}(i)}.
\end{align}
The braiding is
\begin{align}
R_{[\mathbf{c}^\ell]}=\theta_{X_2}^{-1}\cdot\bigoplus\limits_{w\in\mathbb{I}}\zeta_{2(q+1)}^{(\ell+2h+1)w}{\rm id}.
\end{align}

\subsubsection{$f^\ell=a_1^{\epsilon_1}a_2^{\epsilon_2}$ with $\epsilon_1,\epsilon_2\in\{\pm 1\}$}

By Lemma \ref{lem:multiply} (i) and Lemma \ref{lem:real}, $\tilde{\mathbf{y}}_1,\tilde{\mathbf{y}}_2$ are diagonal, so $a_j=f^{\ell_j}$, $\ell_j\in[1,2h]$, and $X_j=C_{\ell_j}^{w_j}$, $j=1,2$, with $\epsilon_1\ell_1+\epsilon_2\ell_2\equiv\ell\pmod{q+1}$.
Now
\begin{align*}
X_{\mathbf{c}^\ell}=\mathbb{C}\langle\mathbf{c}^{\epsilon_1\ell_1}\rangle\otimes\mathbb{C}\langle\mathbf{c}^{\epsilon_2\ell_2}\rangle,
\end{align*}
and we have
\begin{align}
X_{[\mathbf{c}^\ell]}\cong C_{\ell}^{\epsilon_1w_1+\epsilon_2w_2}, \qquad \mathbf{c}^{\epsilon_1\ell_1}\otimes\mathbf{c}^{\epsilon_2\ell_2}\mapsto \mathbf{c}^\ell.
\end{align}
Obviously,
\begin{align}
R_{[\mathbf{c}^\ell]}=\zeta_{q+1}^{\epsilon_1\epsilon_2\ell_1w_2}{\rm id}.
\end{align}

\subsection{$X_{[\mu\mathbf{e}]}$}

Due to that each element is conjugate to its inverse, the non-triviality of $X_{[\mu\mathbf{e}]}$ requires $\mathbf{x}_2=\mu\mathbf{x}_1$.

Suppose
\begin{align*}
X_{[\mu\mathbf{e}]}\cong
n_0\mathbf{1}\oplus m_0V\oplus\bigoplus\limits_{\sigma=1}^{2h-1}n_\sigma W_{\sigma}\oplus
\bigoplus\limits_{\phi=1}^{2h}m_\phi X_{\phi}\oplus n'W'\oplus n''W''\oplus m'X'\oplus m''X''.
\end{align*}
Evaluating at $\pm\mathbf{e}$,  $\pm\mathbf{b}_{\varepsilon}$, $\mathbf{a}^{k'}$ ($k'\in[1,2h-1]$), $\mathbf{c}^{\ell}$ ($\ell\in[1,2h]$) gives respectively
\begin{align*}
n_0+qm_0+(q+1)\sum\limits_{\sigma=1}^{2h-1}n_\sigma(\pm 1)^{\sigma}+(q-1)\sum\limits_{\phi=1}^{2h}m_\phi(\pm1)^{\phi}+(2h+1)(n'+n'') \\
\pm 2h(m'+m'')={\rm tr}_1, \\
n_0+\sum\limits_{\sigma=1}^{2h-1}n_\sigma(\pm1)^\sigma-\sum\limits_{\phi=1}^{2h}m_\phi(\pm1)^\phi
+s_{\varepsilon}n'+s_{-\varepsilon}n''\mp(s_{\varepsilon}m'+s_{-\varepsilon}m'')={\rm tr}_2, \\
n_0+m_0+\sum\limits_{\sigma=1}^{2h-1}(\zeta_{q-1}^{\sigma k'}+\zeta_{q-1}^{-\sigma k'})n_\sigma+(-1)^{k'}(n'+n'')={\rm tr}_3, \\
n_0-m_0-\sum\limits_{\phi=1}^{2h}(\zeta_{q+1}^{\phi\ell}+\zeta_{q+1}^{-\phi\ell})m_\phi-(-1)^{\ell}(m'+m'')={\rm tr}_4,
\end{align*}
with ${\rm tr}_1,{\rm tr}_2,{\rm tr}_3,{\rm tr}_4$ depending on $X_1,X_2$.

\subsubsection{$X_1=A_{k}^{\mu u+d}, X_2=A_{k_\mu}^{u}$}

No representative of a nontrivial conjugacy class fixes any component, except for $\mathbf{a}^{k'}$, which fixes $\mathbb{C}\langle\mathbf{a}^k\otimes\mathbf{a}^{2h\acute{\mu}-k}\rangle$ and $\mathbb{C}\langle\mathbf{a}^{-k}\otimes\mathbf{a}^{k-2h\acute{\mu}}\rangle$, giving trace $\zeta_{q-1}^{k'd}+\zeta_{q-1}^{-k'd}$.

$${\rm tr}_1=(\pm 1)^{d}q(q+1), \qquad {\rm tr}_2={\rm tr}_4=0, \qquad {\rm tr}_3=\zeta_{q-1}^{k'd}+\zeta_{q-1}^{-k'd}.$$

If $\nu=+$, i.e., $2\mid d$, then we solve out
\begin{align}
&n'=n''=1+\delta^{2h}_{|d|}, && n_0=\delta_{d}^{0}, && n_\sigma=2\delta_{2\mid\sigma}+\delta_{|d|}^\sigma,  \\
&m'=m''=0, && m_0=2+\delta_{d}^{0}, && m_\phi=2\delta_{2\mid\phi}.
\end{align}
When $d=0$, we have an isomorphism
\begin{align*}
X_{[\mu\mathbf{e}]}\cong \mathbf{1}\oplus\oplus^3V\oplus\bigoplus\limits_{\sigma'=1}^{h-1}\oplus^2W_{2\sigma'}\oplus\bigoplus\limits_{\phi'=1}^{h}\oplus^2X_{2\phi'}\oplus W'\oplus W'', \\
\mathbf{a}^k\otimes\mathbf{a}^{2h\acute{\mu}-k}\mapsto 1\oplus(\xi^{\mathbf{a}}_{1},\xi^{\mathbf{a}}_{2},\xi^{\mathbf{a}}_{3})\oplus\bigoplus\limits_{\sigma'=1}^{h-1}(\xi^{\mathbf{a}}_{1},\xi^{\mathbf{a}}_{2})
\oplus\bigoplus\limits_{\phi'=1}^{h}(\xi^{\mathbf{a}}_{1},\xi^{\mathbf{a}}_{2})\oplus\xi^{\mathbf{a}}\oplus\xi^{\mathbf{a}},
\end{align*}
and then use this to find the braiding:
\begin{align}
R_{[\mu\mathbf{e}]}=\zeta_{q-1}^{\overline{\mu}ku}\left(1\oplus(J^{\acute{\overline{\mu}}}\oplus 1)\oplus\bigoplus\limits_{\sigma'=1}^{h-1}\overline{\mu}^{h}I\oplus\bigoplus\limits_{\phi'=1}^{h}
J^{\acute{\overline{\mu}}-1}\oplus\overline{\mu}^{h}\oplus\overline{\mu}^{h}\right).
\end{align}
When $|d|=2h$, we have an isomorphism
\begin{align*}
X_{[\mu\mathbf{e}]}\cong\oplus^2V\oplus\bigoplus\limits_{\sigma'=1}^{h-1}\oplus^2W_{2\sigma'}\oplus\bigoplus\limits_{\phi'=1}^{h}\oplus^2X_{2\phi'}\oplus\oplus^2W'\oplus\oplus^2W'', \\
\mathbf{a}^k\otimes\mathbf{a}^{2h\acute{\mu}-k}\mapsto (\xi^{\mathbf{a}}_1,\xi^{\mathbf{a}}_2)\oplus\bigoplus\limits_{\sigma'=1}^{h-1}(\xi^{\mathbf{a}}_1,\xi^{\mathbf{a}}_2)
\oplus\bigoplus\limits_{\phi'=1}^{h}(\xi^{\mathbf{a}}_1,\xi^{\mathbf{a}}_2)\oplus(\xi^{\mathbf{a}}_1,\xi^{\mathbf{a}}_2)\oplus(\xi^{\mathbf{a}}_1,\xi^{\mathbf{a}}_2),
\end{align*}
and then use this to find the braiding:
\begin{align}
R_{[\mu\mathbf{e}]}=\zeta_{q-1}^{\overline{\mu}ku}\left(J^{\acute{\overline{\mu}}}\oplus \bigoplus\limits_{\sigma'=1}^{h-1}\overline{\mu}^{h}I\oplus\bigoplus\limits_{\phi'=1}^{h}
J^{\acute{\overline{\mu}}}\oplus\overline{\mu}^{h}I\oplus\overline{\mu}^{h}I\right).
\end{align}
In the remaining case, $|d|\ne 0,h$,
$$X_{[\mu\mathbf{e}]}\cong\oplus^2V\oplus\oplus^3W_{|d|}\oplus\bigoplus\limits_{\sigma'\in[1,h-1]\atop 2\sigma'\ne |d|}\oplus^2W_{2\sigma'}\oplus\bigoplus\limits_{\phi'=1}^{h}\oplus^2X_{2\phi'}\oplus W'\oplus W'',$$
determined similarly, and the braiding is
\begin{align}
R_{[\mu\mathbf{e}]}=\zeta_{q-1}^{\overline{\mu}ku}\cdot{\rm id}.
%\left(I\oplus 1\oplus\bigoplus\limits_{\sigma=1}^{\overline{q}-1}I\oplus\bigoplus\limits_{\phi=1}^{\overline{q}}I\oplus 1\oplus 1\right).
\end{align}

If $\nu=-$, i.e., $2\nmid d$, then we can solve out
\begin{align}
&n'=n''=0, &&n_0=0, &&n_\sigma=2\delta_{2\nmid\sigma}+\delta_{|d|}^\sigma, \\
&m'=m''=1, &&m_0=0, &&m_\phi=2\delta_{2\nmid\phi},
\end{align}
and the isomorphism
$$X_{[\mu\mathbf{e}]}\cong\oplus^3W_{|d|}\oplus\bigoplus\limits_{\sigma'\in[1,h-1]\atop 2\sigma'-1\ne|d|}\oplus^2W_{2\sigma'-1}\oplus\bigoplus\limits_{\phi'=1}^{h}\oplus^2X_{2\phi'-1}\oplus X'\oplus X''$$
can be determined similarly, and the braiding is also
\begin{align}
R_{[\mu\mathbf{e}]}=\zeta_{q-1}^{\overline{\mu}ku}\cdot{\rm id}.
\end{align}

\subsubsection{$X_1=B_{\mu_1,\varepsilon}^{\nu_1,v_1}, X_2=B_{\mu_2,\varepsilon}^{\nu_2,v_2}$, with $\mu_1\mu_2=\mu$}

\begin{align*}
{\rm tr}_1=(\pm 1)^{\nu_1+\nu_2}\frac{1}{2}(q^2-1), \qquad
{\rm tr}_2=(\pm 1)^{\nu_1+\nu_2}g_\varepsilon, \qquad
{\rm tr}_3={\rm tr}_4=0.
\end{align*}
Here
$$g_{\mu}=\sum\limits_{k=1}^{2h}\zeta_p^{{\rm tr}(e^{2k+\acute{\mu}}(v_1-v_2))},$$
and the second equation follows from that $\mathbf{b}_\varepsilon$ acts on $\mathbb{C}\langle\tilde{\mathbf{a}}^{2k}\rangle\otimes\mathbb{C}\langle\tilde{\mathbf{a}}^{2k+2h}\rangle$ as multiplication by $\zeta_p^{{\rm tr}(e^{\acute{\varepsilon}-2k}(v_1-v_2))}$.

Suppose $v_1=v_2$, then $g_{\pm}=2h$. If $\nu_1=\nu_2$, then we can solve out
\begin{align}
n_0=m_0=n'=n''=1, \qquad n_\sigma=2\delta_{2\mid\sigma}, \qquad m_\phi=m'=m''=0,
\end{align}
Note that $\mathbf{b}_\varepsilon^{-1}=\mathbf{a}^{h}.\mathbf{b}_\varepsilon$, instead of using $\mathbf{j}$.
So
\begin{align}
R_{[\mu\mathbf{e}]}=\mu_2\nu_2\zeta_p^{-{\rm tr}(e^{\acute{\varepsilon}} v_2)}\left(1\oplus1\oplus\bigoplus\limits_{\sigma'=1}^{h}(-1)^{h}I\oplus
(-1)^{h}\oplus(-1)^{h}.\right)
\end{align}
If $\nu_1\ne\nu_2$, then
\begin{align}
n_0=m_0=n'=n''=m'=m''=m_\phi=0, \qquad n_\sigma=2\delta_{2\nmid\sigma},
\end{align}
and \begin{align}
R_{[\mu\mathbf{e}]}=\mu_2\nu_2\zeta_p^{-{\rm tr}(e^{\acute{\varepsilon}}v_2)}\bigoplus\limits_{\sigma'=1}^{h}(-1)^{h}I.
\end{align}

Now suppose $v_1\ne v_2$.
If $\nu_1=\nu_2$, then we can solve out
\begin{align}
&n_0=0, && n_\sigma=\delta_{2\mid\sigma}, && (n',n'')=\begin{cases} (1,0), &g_+=s_+-1, \\ (0,1), &g_+=s_--1. \end{cases}  \\
&m_0=1, && m_{\phi}=\delta_{2\mid\phi}, && m'=m''=0.
\end{align}
The braiding is
\begin{align}
R_{[\mu\mathbf{e}]}=\mu_2\nu_2\zeta_p^{-{\rm tr}(e^{\acute{\varepsilon}}v_2)}{\rm id}.
\end{align}
If $\nu_1\ne\nu_2$, then
\begin{align}
&n_0=0, && n_\sigma=\delta_{2\nmid\sigma}, && n'=n''=0, \\
&m_0=0, && m_{\phi}=\delta_{2\nmid\phi}, && (m',m'')=\begin{cases} (1,0), &g_+=s_+-1, \\ (0,1), &g_+=s_--1. \end{cases}
\end{align}
The braiding is similar.

\subsubsection{$X_1=C_{\ell}^{\mu w+d}$, $X_2=C_{\ell^\mu}^{w}$}

\begin{align*}
{\rm tr}_1=(\pm 1)^{d}q(q-1), \qquad  {\rm tr}_2={\rm tr}_3=0, \qquad {\rm tr}_4=\zeta_{q+1}^{\ell'd}+\zeta_{q+1}^{-\ell'd}.
\end{align*}

If $\nu=+$, i.e., $2\mid w_1-w_2$, then we can solve out
\begin{align}
&n'=n''=1, &&n_0=\delta_{d}^{0}, && n_\sigma=2\delta_{2\mid\sigma}, \\
&m'=m''=0, &&m_0=2-\delta_{d}^{0}, &&m_\phi=2\delta_{2\mid\phi}-\delta_{|d|}^\phi.
\end{align}

When $d=0$, we have
\begin{align*}
&X_{[\mu\mathbf{e}]}\cong\mathbf{1}\oplus V\oplus\bigoplus\limits_{\sigma'=1}^{h-1}\oplus^2W_{2\sigma'}\oplus\bigoplus\limits_{\phi'=1}^{h}\oplus^2X_{2\phi'}\oplus W'\oplus W'', \\
&\mathbf{c}^\ell\otimes\mathbf{c}^{(2h+1)\acute{\mu}-\ell}\mapsto\xi^{\mathbf{c}}\oplus\xi^{\mathbf{c}}\oplus\bigoplus\limits_{\sigma=1}^{h-1}(\xi^{\mathbf{c}}_1\oplus\xi^{\mathbf{c}}_2)\oplus
\bigoplus\limits_{\phi=1}^{h}(\xi^{\mathbf{c}}_1\oplus\xi^{\mathbf{c}}_2)\oplus \xi^{\mathbf{c}}\oplus \xi^{\mathbf{c}},
\end{align*}
and the braiding is
\begin{align}
R_{[\mu\mathbf{e}]}=\zeta_{q+1}^{\overline{\mu}\ell w}\left(1\oplus 1\oplus\bigoplus\limits_{\sigma'=1}^{h-1}\overline{\mu}^{h}I\oplus
\bigoplus\limits_{\phi=1}^{h}J^{\acute{\overline{\mu}}}\oplus\overline{\mu}^{h}\oplus\overline{\mu}^{h}\right).
\end{align}
When $d\ne 0$, the isomorphism can be determined similarly, and $R_{[\mu\mathbf{e}]}=\zeta_{q+1}^{\overline{\mu}\ell w}{\rm id}$.

If $\nu=-$, then we can solve out
\begin{align}
&n_0=0, &&n'=n''=0, &&n_{\sigma}=2\delta_{2\nmid\sigma}, \\
&m_0=0, &&m'=m'=1-\delta_{|d|}^{2h+1}, && m_\phi=2\delta_{2\nmid\phi}-\delta_{|d|}^{2h+1}.
\end{align}
When $|d|=2h+1$,
\begin{align*}
X_{[\mu\mathbf{e}]}\cong\bigoplus\limits_{\sigma'=1}^{h}\oplus^2W_{2\sigma'-1}\oplus\bigoplus\limits_{\phi'=1}^{h}\oplus^2X_{2\phi'-1}, \\
\mathbf{c}^\ell\otimes\mathbf{c}^{(2h+1)\acute{\mu}-\ell}\mapsto\bigoplus\limits_{\sigma'=1}^{h}(\xi^{\mathbf{c}}_1,\xi^{\mathbf{c}}_2)\oplus\bigoplus\limits_{\phi'=1}^{h}(\xi^{\mathbf{c}}_1,\xi^{\mathbf{c}}_2),
\end{align*}
and the braiding is
\begin{align}
R_{[\mu\mathbf{e}]}=\zeta_{q+1}^{\overline{\mu}\ell w}\left(\bigoplus\limits_{\sigma'=1}^{h}\overline{\mu}^{h}\oplus
\bigoplus\limits_{\phi'=1}^{h}J^{\acute{\overline{\mu}}}\right).
\end{align}
When $|d|\ne 2h+1$,
\begin{align*}
X_{[\mu\mathbf{e}]}\cong\bigoplus\limits_{\sigma'=1}^{h}\oplus^2W_{2\sigma'-1}\oplus X_{|d|}\oplus\bigoplus\limits_{\phi'\in[1,h]\atop 2\phi'-1\ne d}\oplus^2X_{2\phi'-1}\oplus X'\oplus X'',
\end{align*}
the isomorphism determined similarly, and the braiding is $R_{[\mu\mathbf{e}]}=\zeta_{q+1}^{\overline{\mu}\ell w}{\rm id}$.
%\begin{align}
%\mathbf{c}^\ell\otimes\mathbf{c}^{\acute{\mu}(2\overline{q}+1)-\ell}\mapsto\bigoplus\limits_{\sigma'=1}^{\overline{q}}(\xi_1,\xi_2)\oplus
%\bigoplus\limits_{\phi'=1}^{\overline{q}}(\xi_1,\xi_2),
%\end{align}

%\section{Application}

%Compute %$\#\hom(\pi_1(S^3-N(K)),\Gamma)$
%the DW invariant for double twist knots $K$.

\end{document}